\documentclass[11pt]{article}
\usepackage{amsmath}
\usepackage{dsfont}
\usepackage{mathrsfs}
\usepackage{amsmath,amssymb}
\usepackage{amsfonts}
\usepackage{hyperref}
\usepackage{amsthm}
\usepackage{graphicx}
\usepackage{subfigure}
\usepackage{xcolor}

\usepackage{bbm}
\usepackage{mathrsfs}
\usepackage{txfonts}
\usepackage{color}
\usepackage{pict2e}
\usepackage{palatino,epsfig,latexsym}
\usepackage{epsf,xy,epic,amscd}

\theoremstyle{plain}
\newtheorem{theorem}{Theorem}[section]
\newtheorem{lemma}[theorem]{Lemma}

\newtheorem{proposition}[theorem]{Proposition}

\newtheorem{Bounded Diameter Lemma}[theorem]{Bounded Diameter Lemma}
\theoremstyle{definition}
\newtheorem{definition}[theorem]{Definition}

\newtheorem{remark}[theorem]{Remark}

\newcommand{\Hmm}[1]{\leavevmode{\marginpar{\tiny%
$\hbox to 0mm{\hspace*{-0.5mm}$\leftarrow$\hss}%
\vcenter{\vrule depth 0.1mm height 0.1mm width \the\marginparwidth}%
\hbox to
0mm{\hss$\rightarrow$\hspace*{-0.5mm}}$\\\relax\raggedright #1}}}

\def\dist{{\text{dist}}}

\hfuzz=\maxdimen
\DeclareFixedFont{\Acknowledgment}{OT1}{cmr}{bx}{n}{14pt}
\begin{document}

\title{Circle patterns with obtuse exterior intersection angles}
\author{Ze Zhou}
\date{}


\maketitle

\begin{abstract}
Thurston's Circle Pattern Theorem studies existence and rigidity of circle patterns  of a given combinatorial type and the given non-obtuse exterior intersection angles. Using topological degree theory, variational principle, Teichm\"{u}ller theory, and Sard's Theorem, this paper generalizes Circle Pattern Theorem to the case of obtuse exterior intersection angles.


\end{abstract}

\section{Introduction}
  The patterns of circles were introduced as useful tools to study  hyperbolic 3-manifolds by Thurston \cite{Thurston}. He also conjectured that, under a procedure of refinement, the hexagonal circle packings converge to the classical Riemann mapping \cite{Thurston2}. In 1987 this conjecture was resolved by Rodin-Sullivan \cite{Rodin-Sullivan}. Over the past decades, circle patterns (packings) have been bridging combinatorics \cite{Schramm,Liu-Zhou}, discrete and computational geometry \cite{Dai-Gu-Luo,Stephenson}, minimal surfaces \cite{Bobenko-Hoffman-Springborn} and others.

Let $S$ be an oriented closed surface and  $\mathcal T$ be a triangulation of $S$ with the sets of vertices, edges and triangles $V,E,F$. Suppose that $S$ is equipped with a constant curvature metric $\mu$. A circle pattern $\mathcal P$ on $(S,\mu)$ is a collection of oriented circles. Say $\mathcal P$ is $\mathcal T$-type if there exists a geodesic triangulation $\mathcal T(\mu)$ of $(S,\mu)$ with the following properties: $(i)$ $\mathcal T(\mu)$ is isotopic to $\mathcal T$; $(ii)$ the vertices of $\mathcal T(\mu)$ coincide with the centers of  circles in $\mathcal P$. In this paper we focus on these circle patterns $\mathcal P=\{C_v: v\in V\}$ such that $C_{u}, C_{w}$ intersect with each other whenever there exists an edge between $u$ and $w$. Then we have the exterior intersection angle $\Theta(e)\in[0,\pi)$ for any $e\in E$.  One refers to Stephenson's monograph
\cite{Stephenson} for more background.

Given a function $\Theta: E\to[0,\pi)$ defined on the edge set of $\mathcal T$, let us consider the following question: Does there exist a $\mathcal T$-type circle pattern whose exterior intersection angle function is given by $\Theta $? If it does, to what extent is the circle pattern unique? A celebrated answer to this question is the following Circle Pattern Theorem due to Thurston \cite[Chap. 13]{Thurston}.

\begin{theorem}[Thurston]\label{T-1-1}
Let $\mathcal T$ be a triangulation of an oriented closed surface $S$ of genus $g>0$. Suppose that $\Theta: E \to [0,\pi/2]$ is a function satisfying the following conditions:
\begin{itemize}
 \item[$(i)$] If the edges $e_1,e_2,e_3$ form a null-homotopic closed curve in $S$, and if $\sum_{i=1}^3 \Theta(e_i)\geq \pi$, then these edges form the boundary of a triangle of $\mathcal T$.
\item[$(ii)$] If the edges $e_1,e_2,e_3,e_4$ form a null-homotopic closed curve in $S$ and if $\sum_{i=1}^4 \Theta(e_i)=2\pi$, then these edges form the boundary of the union of two adjacent triangles.
\end{itemize}
Then there exists a constant curvature (equal to $0$ for $g=1$ and equal to $-1$ for $g>1$) metric $\mu$ on $S$ such that $(S,\mu)$ supports a $\mathcal T$-type circle pattern $\mathcal P$ with the exterior intersection angles given by $\Theta$. Moreover,  the pair $(\mu,\mathcal P)$ is unique up to isometry if $g>1$, and up to similarity if $g=1$.
\end{theorem}

There are many results relating to Circle Pattern Theorem. See, for example, the works of Andreev \cite{Andreev}, Marden-Rodin \cite{Marden-Rodin}, Colin de Verdi\`{e}re \cite{Colin}, Chow-Luo \cite{Chow-Luo} and others. Meanwhile, a natural problem arises: can we relax the requirement of non-obtuse angles in the above theorem? So far few progress has been made, except for some cases considered by Rivin \cite{Rivin1,Rivin2}, Bao-Bonahon \cite{Bao-Bonahon}, Bobenko-Springborn \cite{Bobenko-Springborn} and Schlenker \cite{Schlenker}, respectively. The purpose of this paper is to establish some general results.

First let us explain some of our terminologies. A closed (not necessarily simple) curve $\gamma$ in $S$ is called a pseudo-Jordan curve, if the complement $S\setminus \gamma$ contains a simply connected component whose boundary is equal to $\gamma$. For a pseudo-Jordan curve $\gamma$ in $S$, an enclosing vertex set of $\gamma$  consists of all vertices covered by
$\Omega$, where $\Omega$ is any simply connected component of $S\setminus\gamma$ such that $\partial\Omega=\gamma$. A pseudo-Jordan curve is said to be non-vacant if one of its enclosing vertex sets is non-empty.

\begin{theorem}\label{T-1-2}
Let $\mathcal T$ be a triangulation of an oriented closed surface $S$ of genus $g>1$. Suppose that $\Theta: E \to [0,\pi)$ is a function satisfying the following conditions:
\begin{itemize}
\item[\emph{\textbf{(C1)}}]If the edges $e_1,e_2,e_3$ form the boundary of a triangle of $\mathcal T$, and if $
 \sum_{i=1}^3 \Theta(e_i)>\pi$, then
$\Theta(e_1)+\Theta(e_2)<\pi+\Theta(e_3),$ $\Theta(e_2)+\Theta(e_3)<\pi+\Theta(e_1),$ $\Theta(e_3)+\Theta(e_1)<\pi+\Theta(e_2)$.
\item[\emph{\textbf{(C2)}}]If the edges $e_1,e_2,\cdots,e_s$ form a non-vacant pseudo-Jordan curve in $S$, then
 $\sum_{i=1}^s\Theta(e_i)<(s-2)\pi$.
\end{itemize}
Then there exists a hyperbolic metric $\mu$ on $S$ such that  $(S,\mu)$  supports a $\mathcal T$-type  circle pattern $\mathcal P$ with the exterior intersection angles  given by $\Theta$.
\end{theorem}

\begin{remark}\label{R-1-3}
Under the assumption that every exterior intersection angle is non-obtuse, when $s>4$, it is easy to see $\sum_{i=1}^s\Theta(e_i)\leq s\pi/2<(s-2)\pi$. Therefore, Thurston's conditions imply \textbf{(C2)}.
\end{remark}

\begin{remark}\label{R-1-4}
The contact graph of a circle pattern is a graph which has a vertex for each circle and an edge between two vertices for each intersection component of the corresponding closed disks. For a $\mathcal T$-type circle pattern $\mathcal P$ with acute exterior intersection angles, the contact graph $G(\mathcal P)$ is  isomorphic to the $1$-skeleton of $\mathcal T$. However in obtuse angle cases the statement may not hold. There is a discussion on  relations of $G(\mathcal P)$ and $\mathcal T$ in an early arXiv version of this paper. To relieve the burden of involved details, we do not include it here.
\end{remark}

\begin{theorem}\label{T-1-5}
The  pair $(\mu, \mathcal P)$ in Theorem \ref{T-1-2} is unique up to isometry if \emph{\textbf{(C1)}} is replaced by the following condition:
\begin{itemize}
\item[\emph{\textbf{(R1)}}]If the edges $e_1,e_2,e_3$ form the boundary of a triangle of $\mathcal T$, then $I(e_1)+I(e_2)I(e_3)\geq0,$  $I(e_2)+I(e_3)I(e_1) \geq 0,$ $I(e_3)+I(e_1)I(e_2)\geq 0$, where $I(e_i)=\cos \Theta(e_i)$ for $i=1,2,3$.
\end{itemize}
\end{theorem}

\begin{remark}\label{R-1-6}
The conditions \textbf{(C1)}, \textbf{(R1)} are motivated by spherical trigonometry. To be specific, \textbf{(C1)} is satisfied if and only if $\sum_{i=1}^3 \Theta(e_i)\leq \pi$ or $\Theta(e_1), \Theta(e_2), \Theta(e_3)$ are the three angles of a spherical triangle, and \textbf{(R1)} is satisfied if and only if $\sum_{i=1}^3 \Theta(e_i)\leq \pi$ or each side of the corresponding spherical triangle has length less than or equal to $\pi/2$. In this way one finds that \textbf{(C1)} is strictly wider than \textbf{(R1)}. See Proposition \ref{P-2-7} for  details.
\end{remark}

As a consequence of Theorem \ref{T-1-2} and Theorem \ref{T-1-5}, we obtain the following result which is related to the Hyperideal Circle Pattern Theorem due to Schlenker \cite{Schlenker}.
\begin{theorem}\label{T-1-7}
 Let $\mathcal T$ be  a triangulation of an oriented closed surface $S$ of genus $g>1$. Suppose that $\Theta:E\to [0,\pi)$ is a function satisfying
 \[
\sum\nolimits_{i=1}^s \Theta(e_i) \, <\, (s-2)\pi,
 \]
 whenever $e_1,e_2,\cdots,e_s$  form a pseudo-Jordan curve in $S$. Then there exists a hyperbolic metric $\mu$ on $S$ such that  $(S,\mu)$  supports a $\mathcal T$-type circle pattern $\mathcal P$ with the exterior intersection angles given by $\Theta$. Moreover, the pair $(\mu,\mathcal P)$ is unique up to isometry.
\end{theorem}

One may ask the following question: what can be said regarding rigidity  when \textbf{(R1)} is not satisfied? Let $W$ denote the set of functions $\Theta: E\to [0,\pi)$ satisfying \textbf{(C1)}, \textbf{(C2)}. Then $W$ is a convex subset of $[0,\pi)^{|E|}$. Below is a part answer.
\begin{theorem}\label{T-1-8}
 For almost every $\Theta\in W$, there are at most finitely many $\mathcal T$-type circle pattern pairs $(\mu,\mathcal P)$, up to isometry, with the exterior intersection angles given by $\Theta$.
\end{theorem}

The paper is organized as follows: In next section we introduce some properties of three-circle configurations. In Section~\ref{S-3} we prove Theorem~\ref{T-1-2} by using topological degree theory. In Section~\ref{S-4}, applying variational principle, we derive Theorem~\ref{T-1-5}. As a corollary, Theorem \ref{T-1-7} is established. In Section \ref{S-5} we deduce Theorem \ref{T-1-8} through a combination of Teichm\"{u}ller theory and Sard's Theorem. The last section contains an appendix concerning some results from manifold theory.

Throughout this paper, we denote by $|\cdot|$ the cardinality of a set, and denote by $\chi(\cdot)$ the Euler characteristic of a manifold.

\section{Preliminaries}\label{S-2}

\subsection{Three-circle configurations}
The following three lemmas played crucial roles in the proof of Circle Patten Theorem. Please refer to \cite{Thurston,Chow-Luo} for more information.

\begin{lemma}\label{L-2-1}
For any three positive numbers $r_i, r_j, r_k$ and  three non-obtuse angles $\Theta_i, \Theta_j, \Theta_k$,
 there exists a configuration of three mutually intersecting circles in hyperbolic geometry, unique up to isometry, having radii $r_i, r_j, r_k$  and meeting in exterior intersection angles $\Theta_i, \Theta_j, \Theta_k$.
\end{lemma}

\begin{figure}[htbp]\centering
\includegraphics[width=0.50\textwidth]{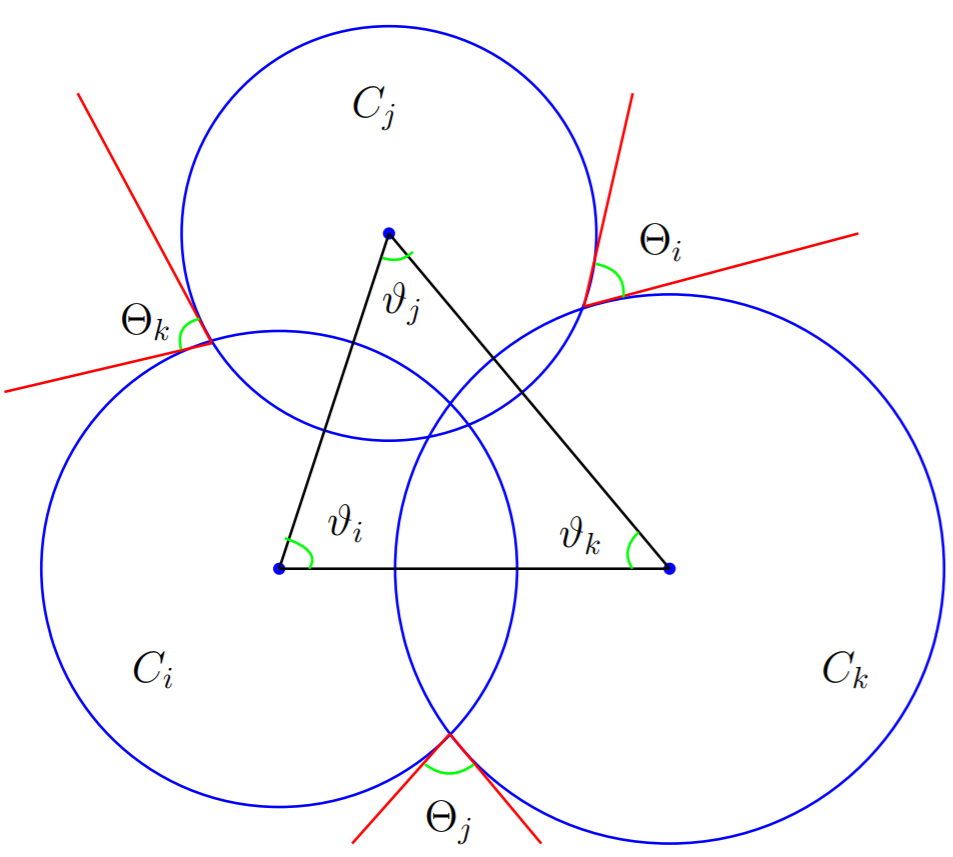}
\caption{A three-circle configuration}
\end{figure}

As in Figure 1, let $\vartheta_i,\vartheta_j,\vartheta_k$ denote the corresponding inner angles of the triangle of centers.

\begin{lemma}\label{L-2-2}
Under the above condition, we have
\[
\frac{\partial\vartheta_i}{\partial r_i}\ <\ 0,\;\; \frac{\partial\vartheta_j}{\partial r_i}\ \geq \ 0,\,\;\; \frac{\partial(\vartheta_i+\vartheta_j+\vartheta_k)}{\partial r_i}\ <\ 0.
\]
\end{lemma}

\begin{lemma}\label{L-2-3}
Let $\Theta_i,\Theta_j,\Theta_k$ be as above. Then
\begin{equation}\label{Eq-1}
\displaystyle{\lim_{r_i\to+\infty}\vartheta_i\,=\,0,}
\end{equation}
\begin{equation}\label{Eq-2}
\displaystyle{\lim_{(r_i,r_j,r_k)\to(0,a,b)}\vartheta_i\,=\,\pi-\Theta_i,}
\end{equation}
\begin{equation}\label{Eq-3}
\displaystyle{\lim_{(r_i,r_j,r_k)\to(0,0,c)}\vartheta_i+\vartheta_j\,=\,\pi,}
\end{equation}
\begin{equation}\label{Eq-4}
\displaystyle{\lim_{(r_i,r_j,r_k)\to(0,0,0)}\vartheta_i+\vartheta_j+\vartheta_k\,=\,\pi,}
\end{equation}
where $a,b,c$ are positive constants.
\end{lemma}

\subsection{Some new observations}
In search of generalizing Circle Pattern Theorem, one needs to go a step further to study three-circle configurations with obtuse exterior intersection angles. Below are some new observations.
\begin{lemma}\label{L-2-4}
Suppose $\Theta_i, \Theta_j, \Theta_k \in [0,\pi)$ satisfy
\[
\Theta_i+\Theta_j+\Theta_k\ \leq \ \pi
\]
or
\[
\Theta_i+\Theta_j\ <\,\pi+\Theta_k,\;\;\Theta_j+\Theta_k\ <\ \pi+\Theta_i,\;\;\Theta_k+\Theta_i\ <\ \pi+\Theta_j.
\]
For any three positive numbers $r_i,r_j,r_k$, there exists a configuration of three mutually intersecting circles in hyperbolic geometry, unique up to isometry, having radii $r_i,r_j,r_k$  and meeting in exterior intersection angles
$\Theta_i,\Theta_j,\Theta_k$.
\end{lemma}

\begin{proof}
Set
\[
l_i\ =\ \cosh^{-1}(\cosh r_j\cosh r_k+\cos\Theta_i\sinh r_j\sinh r_k)
\]
and $l_j, l_k$ similarly. It suffices to check that $l_i,l_j,l_k$ satisfy the triangle inequalities. Namely,
\[
\cosh(l_i+l_j)\ >\ \cosh l_k
\]
and
\[
\cosh(l_i-l_j) \,<\, \cosh l_k.
\]
Equivalently, one needs to show
\begin{equation}\label{Eq-5}
\big(\cosh l_i\cosh l_j-\cosh l_k\big)^2 \,<\, \sinh^2 l_i\sinh^2 l_j\,=\,(\cosh^2 l_i-1)(\cosh^2 l_j-1).
\end{equation}
To simplify notations, for $m=i,j,k$, set
\[
a_m\ =\ \cosh r_m,\;\;x_m\ =\ \sinh r_m.
\]
Then
\begin{equation}\label{Eq-6}
\cosh l_i\,=\,a_ja_k+\cos\Theta_ix_jx_k.
\end{equation}
Substituting (\ref{Eq-6}) into (\ref{Eq-5}), we need to prove
\begin{equation}\label{Eq-7}
\begin{split}
&\sin^2\Theta_ix_j^2x^2_k+\sin^2\Theta_jx^2_k x_i^2+\sin^2\Theta_kx^2_ix^2_j
+(2+2\cos\Theta_i\cos\Theta_j\cos\Theta_k)x_i^2x_j^2x_k^2\\
&\quad+2\lambda_{ijk}a_ja_kx_jx_kx_i^2+2\lambda_{jki}a_ka_ix_kx_ix_j^2+2\lambda_{kij}a_ia_jx_ix_jx_k^2\,>\,0,
\end{split}
\end{equation}
where
\[
\lambda_{ijk}\ =\ \cos\Theta_i+\cos\Theta_j\cos\Theta_k.
\]
Now we divide the proof into the following two cases:
\begin{itemize}
\item[$(\mathrm{I})$] $\Theta_i+\Theta_j+\Theta_k\leq \pi$. Then
\begin{equation}\label{Eq-8}
\begin{aligned}
\lambda_{ijk}\,&=\,\cos\Theta_i+\cos(\Theta_j+\Theta_k)+\sin\Theta_j\sin\Theta_k\\
&\geq\,2\cos\frac{\Theta_i+\Theta_j+\Theta_k}{2}\cos\frac{\Theta_i-\Theta_j-\Theta_k}{2}\\
&\geq\,0.
\end{aligned}
\end{equation}
Similarly,
\[
\lambda_{jki}\,\geq\, 0,\;\;\lambda_{kij}\,\geq\, 0.
\]
Meanwhile,
\[
2+2\cos\Theta_i\cos\Theta_j\cos\Theta_k\,>\,0.
\]
Thus we deduce  (\ref{Eq-7}).

\item[$(\mathrm{II})$] $\Theta_i+\Theta_j+\Theta_k >\pi$ and
    $\Theta_i+\Theta_j<\pi+\Theta_k,$ $\Theta_j+\Theta_k<\pi+\Theta_i,$ $\Theta_k+\Theta_i<\pi+\Theta_j$. Then there exists
 a spherical triangle with angles $\Theta_i,\Theta_j,\Theta_k$. Denote by $\phi_i,\phi_j,\phi_k$ the lengths of sides opposite to $\Theta_i,\Theta_j,\Theta_k$, respectively. By the second cosine law of spherical triangles,
\[
\cos \phi_k\ =\ \frac{\cos\Theta_k+\cos\Theta_i\cos\Theta_j}{\sin\Theta_i\sin\Theta_j}.
\]
Therefore,
\begin{equation}\label{Eq-9}
\lambda_{kij}\,=\,\cos\Theta_k+\cos\Theta_i\cos\Theta_j\,=\, \cos\phi_k\sin\Theta_i\sin\Theta_j.
\end{equation}
Note that  (\ref{Eq-7}) is equivalent to
\begin{equation}\label{Eq-10}
\begin{split}
&\sin^2\Theta_ia_i^2x_j^2x^2_k+\sin^2\Theta_ja_j^2x^2_k x_i^2+\sin^2\Theta_ka_k^2x^2_ix^2_j
-\zeta_{ijk}x_i^2x_j^2x_k^2\\
&\quad+2\lambda_{ijk}a_ja_kx_jx_kx_i^2+2\lambda_{jki}a_ka_ix_kx_ix_j^2+2\lambda_{kij}a_ia_jx_ix_jx_k^2\;>\;0,
\end{split}
\end{equation}
where
\[
\begin{aligned}
\zeta_{ijk}\,&=\,\sin^2\Theta_i+\sin^2\Theta_j+\sin^2\Theta_k-(2+2\cos\Theta_i\cos\Theta_j\cos\Theta_k)\\
         \,&=\,\sin^2\Theta_i\sin^2\Theta_j-(\cos\Theta_k+\cos\Theta_i\cos\Theta_j)^2\\
         \,&=\,\sin^2\Theta_i\sin^2\Theta_j-\cos^2\phi_k\sin^2\Theta_i\sin^2\Theta_j\\
         \,&=\,\sin^2\phi_k\sin^2\Theta_i\sin^2\Theta_j.\\
\end{aligned}
\]
Set $y_{i}=\sin\Theta_ia_ix_jx_k,$ $y_{j}=\sin\Theta_ja_jx_kx_i$ and  $y_{k}=\sin\Theta_ka_kx_ix_j$. Substituting (\ref{Eq-9}) into (\ref{Eq-10}), it remains to check
\[
y_{i}^2+y_{j}^2+y_{k}^2+2\cos \phi_iy_{j}y_{k}+2\cos \phi_j y_{k}y_{i}+2\cos \phi_ky_{i}y_{j}\,>\,\zeta_{ijk}x_i^2x_j^2x_k^2.
\]
Completing the square gives
\[
\begin{aligned}
&\, y_{i}^2+y_{j}^2+y_{k}^2+2\cos \phi_iy_{j}y_{k}+
2\cos \phi_j y_{k}y_{i}+2\cos \phi_ky_{i}y_{j}\\
=\,&\,(y_{i}+\cos \phi_jy_{k}+\cos\phi_ky_{j})^2
+\sin^2\phi_jy^2_{k}+\sin^2\phi_ky^2_{j}
+2(\cos\phi_i-\cos\phi_j\cos\phi_k)y_{j}y_{k}\\
\geq\, &\, \sin^2\phi_jy^2_{k}+\sin^2\phi_ky^2_{j}+
2(\cos\phi_i-\cos\phi_j\cos\phi_k)y_{j}y_{k}.
\end{aligned}
\]
By the cosine law of spherical triangles, one obtains
\[
\cos\phi_i-\cos\phi_j\cos\phi_k\,=\,\cos\Theta_i\sin\phi_j\sin\phi_k.
\]
It follows that
\[
\begin{aligned}
&\ y_{i}^2+y_{j}^2+y_{k}^2+2\cos \phi_iy_{j}y_{k}+2\cos \phi_j y_{k}y_{i}+2\cos \phi_ky_{i}y_{j}\\
\geq &\,\sin^2\phi_jy^2_{k}+\sin^2\phi_ky^2_{j}+
2\cos\Theta_i\sin\phi_j\sin\phi_ky_{j}y_{k}\\
=&\,(\sin\phi_jy_{k}+\cos\Theta_i\sin\phi_ky_{j})^2
+\sin^2\Theta_i\sin^2\phi_ky^2_{j}\\
\geq &\,\sin^2\Theta_i\sin^2\phi_k \sin^2\Theta_ja_j^2x_k^2x_i^2\\
> &\,\sin^2\phi_k\sin^2\Theta_i\sin^2\Theta_jx_i^2x_j^2x_k^2\\
= &\,\zeta_{ijk}x_i^2x_j^2x_k^2.
\end{aligned}
\]
\end{itemize}
Thus the lemma is proved.
\end{proof}

\begin{remark}\label{R-2-5}
On the other hand, if the triangle inequalities hold for all triples of positive numbers $r_i,r_j,r_k$, then the condition of Lemma \ref{L-2-4} must be satisfied. Namely, the above lemma is in the optimal form.
\end{remark}

\begin{remark}\label{R-2-6}
If $\Theta_i,\Theta_j,\Theta_k\in[0,\pi)$ satisfy $\lambda_{ijk}\geq 0, \lambda_{jki}\geq 0, \lambda_{kij}\geq0$, the above proof can be simplified because (\ref{Eq-7}) automatically holds. In fact the following Proposition \ref{P-2-7} manifests that this condition is stronger than the condition of Lemma \ref{L-2-4}.
\end{remark}

\begin{proposition}\label{P-2-7}
Given $\Theta_i,\Theta_j,\Theta_k\in[0,\pi)$, we have
$\lambda_{ijk}\geq0, \lambda_{jki}\geq0, \lambda_{kij}\geq0$
if and only if one of the following properties holds:
\begin{itemize}
\item[$(i)$] $\Theta_i+\Theta_j+\Theta_k\leq\pi$;
\item[$(ii)$]  $\Theta_i, \Theta_j, \Theta_k$ are the angles of a spherical triangle with each side less than or equal to $\pi/2$.
\end{itemize}
\end{proposition}

\begin{proof}
The "if" part is an immediate consequence of  (\ref{Eq-8}) and (\ref{Eq-9}).

To show the "only if" part, assume $\Theta_i+\Theta_j+\Theta_k>\pi$. It suffices to verify property $(ii)$. First a routine calculation gives
\[
\begin{aligned}
0\,\leq\, \lambda_{ijk}\,
=&\,\cos\Theta_i+\cos\Theta_j\cos\Theta_k\\
=&\,\cos\Theta_i+\cos(\Theta_j-\Theta_k)-\sin\Theta_j\sin\Theta_k\\
=&\,2\cos\frac{\Theta_i+\Theta_j-\Theta_k}{2}\cos\frac{\Theta_i+\Theta_k-\Theta_j}{2}-\sin\Theta_j\sin\Theta_k.
\end{aligned}
\]

If $\Theta_j=0$, one shows $\Theta_i+\Theta_k\leq\pi$, which yields $\Theta_i+\Theta_j+\Theta_k\leq0$. This contradicts to the assumption that $\Theta_i+\Theta_j+\Theta_k>\pi$.

If $\Theta_k=0$, similar arguments lead to a contradiction.

Thus $\Theta_j,\Theta_k\in(0,\pi)$, which implies
\[
\cos\frac{\Theta_i+\Theta_j-\Theta_k}{2}\cos\frac{\Theta_i+\Theta_k-\Theta_j}{2}\,>\,0.
\]
As a result,
\[
\Theta_i+\Theta_j\,<\,\Theta_k+\pi,\;\; \Theta_i+\Theta_k\,<\,\Theta_j+\pi.
\]
Similarly,
\[
\Theta_j+\Theta_k\,<\,\Theta_i+\pi.
\]
Hence there exists a spherical triangle with  angles $\Theta_i,\Theta_j,\Theta_k$. Using (\ref{Eq-9}), it is easy to see each side of this spherical triangle has length less than or equal to $\pi/2$.
\end{proof}

\begin{lemma}\label{L-2-8}
Under the condition of Lemma \ref{L-2-4}, the formulas (\ref{Eq-1}), (\ref{Eq-2}), (\ref{Eq-3}), (\ref{Eq-4}) hold.
\end{lemma}

\begin{proof} The proof of (\ref{Eq-1}) is similar to the related results in Ge-Jiang \cite{Ge-Jiang} and Ge-Xu \cite{Ge-Xu}. Due to the cosine law of hyperbolic triangles, we have
\[
\begin{aligned}
\cos \vartheta_i\,=\, &\,\frac{\cosh l_{j}\cosh l_{k}-\cosh l_{i}}{\sinh l_{j}\sinh l_{k}}\\
=\,&\,\frac{\cosh(l_{j}+l_{k})+\cosh(l_{j}-l_{k})-2\cosh l_{i}}{\cosh(l_{j}+l_{k})-\cosh(l_{j}-l_{k})}\\
=\,&\,\frac{1+A-2B}{1-A},
\end{aligned}
\]
where
 \[
 A\,=\, \frac{\cosh (l_{j}-l_{k})}{\cosh(l_{j}+l_{k})},\;\;B\,=\,\frac{\cosh l_{i}}{\cosh(l_{j}+l_{k})}.
 \]
It suffices to show $A, B\to 0$ as $r_i\to +\infty$. For $m=i,j,k$, setting $c_m=\min\{\cos \Theta_m,0\}$, one obtains $0<1+ c_m\leq 1$. Moreover,
\[
\begin{aligned}
\cosh l_k\,=\,&\cosh r_{i}\cosh r_{j}+\cos\Theta_k\sinh r_{i}\sinh r_{j}\\
\geq\,&\,\cosh r_i\cosh r_j+c_k\sinh r_{i}\sinh r_{j}\\
\geq\,&\,\cosh r_i\cosh r_j+c_k\cosh r_{i}\cosh r_{j}\\
=\,&\,(1+c_k)\cosh r_i\cosh r_j\\
\geq\,&\,(1+c_k)\cosh r_i.
\end{aligned}
\]
Similarly,
\[
\cosh l_j\,\geq\,(1+c_j)\cosh r_k\cosh r_i\,\geq\, (1+c_j)\cosh r_i.
\]
As a result,
\[
0\,<\,A\,\leq\, \frac{1}{\min\{\cosh l_{j},\cosh l_{k}\}}
\,\leq\, \frac{1}{(1+c_k)(1+c_j)\cosh r_i}.
\]
Hence $A\to 0$ as $r_i\to +\infty$. Meanwhile,
\[
\cosh l_i \,=\, \cosh r_{j}\cosh r_{k}+\cos\Theta_i\sinh r_{j}\sinh r_{k}\,\leq\, 2\cosh r_j\cosh r_k.
\]
It follows that
\[
0\ <\ B\, \leq\, \frac{\cosh l_{i}}{\cosh l_{j}\cosh l_{k}}\,\leq\, \frac{2}{(1+c_k)(1+c_j)\cosh^2 r_i}.
\]
Then $B\to 0$ as $r_i\to +\infty$, which concludes  (\ref{Eq-1}).

The formula (\ref{Eq-2}) is derived from a direct computation.

Let us consider (\ref{Eq-3}). By the cosine law of hyperbolic triangles, we have
\[
 \begin{aligned}
0 \ <\  \cos \vartheta_i+\cos \vartheta_j\,=\,&\, \frac{\sinh (l_i+l_j)\big(\cosh l_{k}-\cosh(l_i- l_{j})\big)}{\sinh l_{i}\sinh l_j\sinh l_{k}}\\
\,\leq\,&\,\frac{\sinh (l_i+l_j)(\cosh l_{k}-1)}{\sinh l_{i}\sinh l_j\sinh l_{k}}\\
\,=\,&\,\frac{\sinh (l_i+l_j)\sinh(l_k/2)}{\sinh l_{i}\sinh l_j\cosh( l_{k}/2)}.
\end{aligned}
\]
As $(r_i,r_j,r_k)\to (0,0,c)$, one obtains
\[
l_i\,\to\, c,\;\; l_j\,\to\, c,\;\; l_k\,\to\, 0.
\]
Consequently,
\[
\cos \vartheta_i+\cos \vartheta_j \,\to\, 0,
\]
which yields
\[
\vartheta_i+\vartheta_j\, \to\, \pi.
\]

It remains to prove (\ref{Eq-4}). As $(r_i,r_j,r_k)\to (0,0,0)$, the area of the triangle tends to zero. Due to the Gauss-Bonnet formula, one obtains the desired result.
\end{proof}

\begin{remark}\label{R-2-9}
From the above proof, one finds that (\ref{Eq-1}) holds uniformly, no matter how the other two radii behave, even if either or both of them tend to infinity or zero. In addition, regarding $\vartheta_i$ as the function of $r_i, r_j, r_k,\Theta_i,\Theta_j,\Theta_k$, the formula is true as $(r_i,r_j, r_k,\Theta_i,\Theta_j,\Theta_k)$ varies in $\mathbb R_{+}^3\times \Lambda$, where $\Lambda\subset [0,\pi)^3$ is any compact set such that the condition of Lemma \ref{L-2-4} is satisfied.
\end{remark}

\section{Existence}\label{S-3}
To prove Theorem \ref{T-1-2}, a natural strategy is to follow Thurston \cite{Thurston}. Unfortunately, for three-circle configurations with obtuse exterior intersection angles, the result similar to Lemma \ref{L-2-2} may not hold (see Remark \ref{R-4-2}). Part of his method does not work. In order to overcome the difficulty, we will employ the topological degree theory. As a comparison, this approach has the advantage of being independent on rigidity.

\subsection{Thurston's construction} Recall that $S$ is an oriented closed surface of genus $g>1$ and $\mathcal T$ is a triangulation of $S$ with the sets of vertices, edges and triangles $V, E, F$. Let $r\in \mathbb R_{+}^{|V|}$ be a radius vector, which assigns each vertex $v\in V$ a positive number $r(v)$. Then $r$ together with the  function $\Theta: E\to [0,\pi)$ in Theorem \ref{T-1-2}  determines a hyperbolic cone metric structure on $S$ as follows.

For each triangle $\triangle(v_iv_jv_k)$ of $\mathcal T$, one associates it with the hyperbolic triangle determined by the centers of three mutually intersecting circles with  radii
\[
 r(v_i),\;r(v_j),\;r(v_k)
 \]
 and exterior intersection angles
 \[
 \Theta\big([v_i,v_j]\big),\;\Theta\big([v_j,v_k]\big),\;\Theta\big([v_k,v_i]\big).
 \]
 Because $\Theta$ satisfies \textbf{(C1)},  Lemma \ref{L-2-4} implies the above procedure works well.

Gluing these hyperbolic triangles produces a metric surface $\big(S,\mu(r)\big)$ which is locally hyperbolic with possible cone type singularities at the vertices. For each $v\in V$, the vertex curvature $k(v)$ is defined by
\[
k(v)\;=\;2\pi\ -\ \sigma(v),
\]
where $\sigma(v)$  denotes the cone angle at $v$. More precisely, $\sigma(v)$ is equal to the sum of all inner angles having vertex $v$.

Write $k(v)$ as $k(v)(\Theta, r)$. If there exists a radius vector $r_\ast$ such that $k(v)(\Theta, r_\ast)=0$ for all $v\in V$, then it produces a smooth hyperbolic metric on $S$. For every $v\in V$, on $\big(S,\mu(r_\ast)\big)$ drawing the circle centered at $v$ with radius $r_\ast(v)$, one then obtains the demanded circle pattern. Consider the following curvature map
\[
\begin{aligned}
Th(\Theta,\cdot):\quad\qquad &\,\mathbb{R}_{+}^{|V|} \quad
&\,\longrightarrow \,\;\qquad\qquad &\,\mathbb{R}^{|V|} \\
\big(r(v_1),&\,r(v_2),\cdots \big)&\longmapsto\;\quad\,\,\big(k(v_1), &\,k(v_2),\cdots\big).\\
  \end{aligned}
  \]
The major purpose of this section is to show that the origin $\mathrm{O}=(0,0,\cdots,0)\in \mathbb R^{|V|}$ belongs to the image of $Th(\Theta,\cdot)$.

\subsection{Topological degree} Let us make use of the topological degree theory. Specifically, one finds a relatively compact open set $\Lambda\subset \mathbb{R}_{+}^{|V|}$ and determines the degree $\deg(Th(\Theta,\cdot),\Lambda,\mathrm{O})$. Once showing
\[
\deg(Th(\Theta,\cdot),\Lambda,\mathrm{O})\, \neq\, 0,
\]
it follows from Theorem \ref{T-6-8} that $\mathrm{O}$ is in the image of $Th(\Theta,\cdot)$.

We shall compute $\deg(Th(\Theta,\cdot),\Lambda,\mathrm{O})$ via homotopy method. Namely, deform $Th(\Theta,\cdot)$ to another map which is relatively easier to manipulate. For each $t\in[0,1]$, note that the function $\Theta_t=t\Theta$ satisfies \textbf{(C1)}. Applying Thurston's construction, $Th_t(\cdot)=Th(\Theta_t,\cdot)$ is well-defined. Moreover, $Th_t$ forms a  homotopy from $Th(\Theta,\cdot)$ to $Th(0,\cdot)$.

\begin{lemma}\label{L-3-1}
 There exists a relatively compact open set $\Lambda\subset \mathbb{R}_{+}^{|V|}$ such that
 \[
 Th_t\big(\,\mathbb{R}_{+}^{|V|}\setminus\Lambda\,\big)\ \subset\  \mathbb{R}^{|V|}\setminus\{\mathrm{O}\},\;\;\forall\,t\in[0,1].
 \]
\end{lemma}
 \begin{proof}
Let us exhaust $\mathbb R_{+}^{|V|}$ by an increasing sequence of relatively compact open sets $\{\Lambda_n\}$. Assume on the contrary that the lemma is not true. For each $n$, one obtains $t_n\in [0,1]$ and $r_n\in\mathbb R_{+}^{|V|}\setminus \Lambda_n$ satisfying
 \begin{equation}\label{Eq-11}
 k(v)(t_n\Theta,r_n)\,=\,0,\quad \forall\, v\in V.
  \end{equation}
Because $\{\Lambda_n\}$ exhausts $\mathbb R_{+}^{|V|}$,  there exist $v_0\in V$ and $\{r_{n_k}\}$ such that
  \[
  r_{n_k}(v_0)\,\to \,+\infty\quad \text{or} \quad r_{n_k}(v_0)\,\to\, 0.
  \]

In the first case, by Lemma \ref{L-2-8} and Remark \ref{R-2-9}, it follows from formula (\ref{Eq-1}) that
\[
k(v_0)(t_{n_k}\Theta,r_{n_k})\,\to\, 2\pi,
\]
which contradicts to (\ref{Eq-11}).

In the second case,  let $V_0\subset V$ be the set of vertices $v\in V$ for which $r_{n_k}(v)\to 0$. Then $V_0$ is a non-empty subset of $V$. We denote by $S(V_0)$ the union of these $q$-cells $(q=0,1,2)$ of $\mathcal T$ that have at least one vertex in $V_0$, and denote by $Lk(V_0)$ the set of pairs $(e,v)$ an edge $e$ and a vertex $v$ with the following properties:
\[
(i)\ v\in V_0;\quad (ii)\ \partial e\cap V_0=\emptyset;\quad (iii)\ e \ \mathrm{and}\ v \mathrm{\ form\  a\  triangle\ of\ } \mathcal T.
\]
Evidently, $S(V_0)$ is an open set of $S$ and thus is a surface. Suppose that $t_{n_k}$ converges to a number $t_\ast\in[0,1]$. Otherwise, one picks up a convergent subsequence. By the following Proposition \ref{P-3-2},  we have
\[
\sum\nolimits_{v\in V_0}k(v)(t_{n_k}\Theta,r_{n_k})\,\to\, -\sum\nolimits_{(e,v)\in Lk(V_0)}\big(\pi-t_\ast\Theta(e)\big)+2\pi\chi\big(S(V_0)\big).
\]
Combining  with formula (\ref{Eq-11}) gives
\[
0\,=\,-\sum\nolimits_{(e,v)\in Lk(V_0)}\big(\pi-t_\ast\Theta(e)\big)+2\pi\chi\big(S(V_0)\big).
\]
Without loss of generality, assume that $S(V_0)$ is connected. Otherwise, one  considers the connected component of $S(V_0)$. Note that
\[
\chi\big(S(V_0)\big)\,=\,2-2g_0-h_0,
\]
where $g_0, h_0$ are the genus and the number of boundary components of $S(V_0)$. Hence
\begin{equation}\label{Eq-12}
0\,=\,-\sum\nolimits_{(e,v)\in Lk(V_0)}\big(\pi-t_\ast\Theta(e)\big)+2\pi(2-2g_0-h_0).
\end{equation}

If $h_0=0$, then $S(V_0)$ has no boundary component. It follows that
\[
S(V_0)\,=\,S,\;\; V_0\,=\,V,\;\; Lk(V_0)\,=\,\emptyset, \;\;g_0\,=\,g.
\]
Substituting these into (\ref{Eq-12}), one derives
\[
0\,=\,2\pi(2-2g),
\]
which contradicts to the condition $g>1$.

Let $h_0>0$. If $g_0>0$ or $h_0>1$, it is easy to see  (\ref{Eq-12}) leads to a contradiction.

Let $g_0=0$ and $h_0=1$. Then $S(V_0)$ is a simply connected domain in $S$. Suppose that $Lk(V_0)=\{(e_i,v_i)\}_{i=1}^s$. Formula (\ref{Eq-12})  gives
\[
0\,=\,-\sum\nolimits_{i=1}^s\big(\pi-t_\ast\Theta(e_i)\big)+2\pi,
\]
which yields
\[
\sum\nolimits_{i=1}^s \Theta(e_i)\,\geq\, \sum\nolimits_{i=1}^s t_\ast\Theta(e_i)\,=\,(s-2)\pi.
\]
Meanwhile, one finds that $e_1,\cdots,e_s$ form a non-vacant pseudo-Jordan curve. According to \textbf{(C2)}, we have
\[
\sum\nolimits_{i=1}^s \Theta(e_i)\,<\,(s-2)\pi,
\]
which leads to a contradiction.
\end{proof}

\begin{proposition}\label{P-3-2}
Let $S(V_0)$ and $Lk(V_0)$ be as above. Then
\[
\sum\nolimits_{v\in V_0}k(v)(t_{n_k}\Theta,r_{n_k})\,\to\,
 -\sum\nolimits_{(e,v)\in Lk(V_0)}\big(\pi-t_\ast\Theta(e)\big)+2\pi\chi\big(S(V_0)\big).
\]
\end{proposition}
\begin{proof}
For $m=1,2,3$, let $F_m(V_0)$ be the set of triangles  having exactly $m$ vertices in $V_0$. Using Lemma \ref{L-2-8}, it follows from formulas (\ref{Eq-2}), (\ref{Eq-3}) and (\ref{Eq-4}) that
\[
\sum\nolimits_{v\in V_0}k(v)(t_{n_k}\Theta,r_{n_k})\,\to\,2\pi|V_0| -\sum\nolimits_{(e,v)\in Lk(V_0)}\big(\pi-t_\ast\Theta(e)\big)-\pi|F_2(V_0)|-\pi|F_3(V_0)|.
\]
Let $E(V_0), F(V_0)$ denote the sets of edges and triangles having at least one vertex in $V_0$. It is easy to see
\[
|F(V_0)|\,=\,|F_1(V_0)|+|F_2(V_0)|+|F_3(V_0)|.
\]
Meanwhile, note that
\[
|F_1(V_0)|\,=\,|Lk(V_0)|
\]
and
\[
3|F(V_0)|\,=\,2|E(V_0)|+|Lk(V_0)|.
\]
Combining the above relations gives
\[
\begin{aligned}
   &2|V_0|-|F_2(V_0)|-|F_3(V_0)|\\
=\,&2|V_0|-|F(V_0)|+|F_1(V_0)|-\big(|F_1(V_0)|-|Lk(V_0)|\big)\\
=\,&2|V_0|-|F(V_0)|+|Lk(V_0)|+\big(3|F(V_0)|-2|E(V_0)|-|Lk(V_0)|\big)\\
=\,&2\big(|V_0|-|E(V_0)|+|F(V_0)|\big)\\
=\,&2\chi\big((S(V_0)\big).
\end{aligned}
\]
As a result, we have
\[
\sum\nolimits_{v\in V_0}k(v)(t_{n_k}\Theta,r_{n_k})\,\to\, -\sum\nolimits_{(e,v)\in Lk(V_0)}\big(\pi-t_\ast\Theta(e)\big)+2\pi\chi\big(S(V_0)\big).
\]
\end{proof}

\begin{theorem}\label{T-3-2}
Let $Th(\Theta,\cdot)$ and $\Lambda$ be as above. Then
\[
\deg(Th(\Theta,\cdot),\Lambda,\mathrm{O})\ =\ 1.
\]
\end{theorem}
\begin{proof}
It suffices to compute $\deg(Th_0,\Lambda,\mathrm{O})$. By Theorem \ref{T-1-1}, $\Lambda\cap Th_{0}^{-1}(\mathrm{O})$ consists of a unique point. Using Lemma \ref{L-2-2}, one shows that the Jacobian matrix of $Th_{0}(\cdot)$ has positive diagonal entries and strictly diagonally dominant columns. As a result, the determinant of this matrix is positive and the tangent map preserves the orientation. Therefore,
\[
\deg(Th_0,\Lambda,\mathrm{O})\ =\ 1.
\]
Owing to Lemma \ref{L-3-1} and Theorem \ref{T-6-7}, one deduces the conclusion.
\end{proof}

\begin{proof}[\textbf{Proof of Theorem \ref{T-1-2}}]
Because of Theorem \ref{T-3-2} and Theorem \ref{T-6-8}, $\mathrm{O}$ is in the image of the map $Th(\Theta,\cdot)$. Consequently, there exists a hyperbolic metric $\mu$ on $S$ such that $(S,\mu)$ supports a $\mathcal T$-type circle pattern $\mathcal P$ with the exterior intersection angles given by $\Theta$.
\end{proof}

\section{Rigidity}\label{S-4}

\subsection{Three-circle configurations revisited}
For a triple of indices $i,j,k$, recall that $\lambda_{ijk}=\cos\Theta_i+\cos\Theta_j\cos\Theta_k$. Under the condition that $\lambda_{ijk}\geq 0, \lambda_{jki}\geq 0, \lambda_{kij} \geq 0$, it follows from Remark \ref{R-2-6} that the inner angles $\vartheta_i,\vartheta_j,\vartheta_k$ are well-defined functions of  $(r_i,r_j,r_k)$. Our key observation in this section is the following generalization of Lemma \ref{L-2-2}.

\begin{lemma}\label{L-4-1}
Suppose $\Theta_i,\Theta_j,\Theta_k\in [0,\pi)$ satisfy
\[
\lambda_{ijk}\,\geq\,0,\;\; \lambda_{jki}\,\geq\, 0,\;\; \lambda_{kij}\,\geq\,0.
\]
Then
\[
\frac{\partial\vartheta_i}{\partial r_i}\,<\,0,\;\; \sinh r_j\frac{\partial\vartheta_i}{\partial r_j}\,=\,\sinh r_i\frac{\partial\vartheta_j}{\partial r_i}\,\geq\,0,\;\; \frac{\partial(\vartheta_i+\vartheta_j+\vartheta_k)}{\partial r_i} \,<\,0.
\]
\end{lemma}

\begin{proof}
We mention that part of computations here is parallel to these in Colin de Verdi\`{e}re \cite{Colin}, Chow-Luo \cite{Chow-Luo}, Guo-Luo \cite{Guo-Luo}, Guo \cite{Guo}, Xu \cite{Xu} and others. For simplicity, we use the same notations in the proof of Lemma \ref{L-2-4}. By the cosine law of hyperbolic triangles,
\[
\cos\vartheta_i\ =\ \frac{\cosh l_j\cosh l_k-\cosh l_i}{\sinh l_j\sinh l_k}.
\]
Taking the partial derivative with respect to $l_i$ gives
\[
\frac{\partial \vartheta_i}{\partial l_i}\ =\ \frac{\sinh l_i}{\sin \vartheta_i\sinh l_j\sinh l_k}\ =\ \frac{\sinh l_i}{K_{ijk}},
\]
where
\[
K_{ijk}\ =\ \sin\vartheta_i\sinh l_j\sinh l_k.
\]
By the sine law of hyperbolic triangles, it is easy to see
\[
K_{ijk}\ =\ K_{jki}\ =\ K_{kij}\ :=\ K.
\]
Hence
\begin{equation}\label{Eq-13}
\frac{\partial \vartheta_i}{\partial l_i}\ =\ \frac{\sinh l_i}{K}.
\end{equation}
Similarly,
\begin{equation}\label{Eq-14}
\frac{\partial \vartheta_i}{\partial l_k}\ =\ -\frac{\sinh l_i\cos \vartheta_j}{K}.
\end{equation}
Meanwhile,
\begin{equation}\label{Eq-15}
\frac{\partial l_i}{\partial r_j}\,=\, \frac{\cosh r_k\sinh r_j+\cos\Theta_i\cosh r_j\sinh r_k}{\sinh l_i}\,=\, \frac{a_k x_j+\cos\Theta_ia_jx_k}{\sinh l_i}.
\end{equation}
Combining (\ref{Eq-13}), (\ref{Eq-14}) and (\ref{Eq-15}), one obtains
\begin{equation}\label{Eq-16}
\begin{aligned}
\sinh r_j\frac{\partial\vartheta_i}{\partial r_j}
\,=&\,\sinh r_j\,\bigg(\frac{\partial\vartheta_i}{\partial l_i}\frac{\partial l_i}{\partial r_j}+\frac{\partial\vartheta_i}{\partial l_k}\frac{\partial l_k}{\partial r_j}\bigg)\\
=&\, \frac{\sin^2 \Theta_k a_k^2x_i^2x_j^2+(\lambda_{ijk}a_jx_i+\lambda_{jki}a_ix_j)x_ix_jx_k}{K\sinh^2 l_k}.
\end{aligned}
\end{equation}
It follows that
\[
\sinh r_j\frac{\partial\vartheta_i}{\partial r_j}\,=\,\sinh r_i\frac{\partial\vartheta_j}{\partial r_i}\,\geq\,0,
\]
where the equality holds if and only if $\Theta_k=0$ and $\Theta_i+\Theta_j=\pi$.

Fix $r_j, r_k$ and let $r_i$ vary. Then $l_i$ stays constant. Because $\partial\vartheta_j/\partial r_i\geq 0$ and $\partial\vartheta_k/\partial r_i\geq 0$, that means the other two sides move outwards or remain unchanged as $r_i$ increases. Note that these two sides can not stay unchanged simultaneously. Thus the area $\mathrm{Area}(\triangle)$ of the triangle is a strictly increasing function of $r_i$. By the Gauss-Bonnet formula, one derives
\[
\frac{\partial(\vartheta_i+\vartheta_j+\vartheta_k)}{\partial r_i}\,=\,-\frac{\partial}{\partial r_i}\mathrm{Area}(\triangle)\,<\,0.
\]
Furthermore,
\[
\frac{\partial\vartheta_i}{\partial r_i}\,=\,\frac{\partial(\vartheta_i+\vartheta_j+\vartheta_k)}{\partial r_i}-\frac{\partial\vartheta_j}{\partial r_i}-\frac{\partial\vartheta_k}{\partial r_i}\,<\,0.
\]
\end{proof}

\begin{remark}\label{R-4-2}
Assume that $\lambda_{ijk}<0$. For any $t>0$, set
\[
x_i\,=\,-\lambda_{ijk}\big(\,1+|\lambda_{jki}|\,\big)\,t,\;\;
x_j\,=\,\lambda_{ijk}^2\,t,\;\;
x_k\,=\,-3\lambda_{ijk}\big(\,1+|\lambda_{jki}|\,\big)\,t.
\]
Using (\ref{Eq-16}), one finds that $\partial\vartheta_i/\partial r_j<0$ when $t$ is sufficiently small. Thus the condition of Lemma \ref{L-4-1} can not be relaxed further.
\end{remark}

\subsection{Variational principle}
Let us establish rigidity via the variational approach poineered by Colin de Verdi\`{e}re \cite{Colin}. Remind that similar methods have been used by Rivin \cite{Rivin1,Rivin2}, Bobenko-Springborn \cite{Bobenko-Springborn}, Guo-Luo \cite{Guo-Luo}, Guo \cite{Guo} and others in many situations.

Suppose that $V=\{v_1,\cdots,v_{|V|}\}$. To simplify notations, for $i=1,2,\cdots,|V|$, set
\[
r_i\,:=\,r(v_i),\quad k_i\,:=\,k(v_i).
\]
For a vertex $v_i\in V$ and a triangle $\bigtriangleup\in F$ incident to $v_i$, denote by $\vartheta_i^{\triangle}$ the inner angle of the triangle $\triangle$ at $v_i$. Then
\[
k_i\,=\,2\pi-\Big(\,\sum\nolimits_{\triangle\in F(\{v_i\})} \vartheta_i^{\triangle}\,\Big).
\]
Consider the change of variables $u_i=\ln \tanh (r_i/2)$. It is easy to see
\[
\frac{\partial}{\partial u_i}\,=\,\sinh r_i\frac{\partial}{\partial r_i}.
\]
Regarding vertex curvatures as functions of $u$ induces the map
\[
\begin{aligned}
Th(\Theta,\cdot):\quad\qquad &\,\mathbb{R}_{-}^{|V|} \qquad
&\,\longrightarrow \,\qquad \qquad \quad &\,\mathbb{R}^{|V|} \\
\big(u_1,u_2,&\cdots,u_{|V|}\big)&\longmapsto\quad\quad\;\big(k_1, k_2,&\cdots,k_{|V|}\big).\\
 \end{aligned}
\]

\begin{lemma}\label{L-4-3}
Under the condition \emph{\textbf{(R1)}}, the Jacobian matrix of $Th(\Theta,\cdot)$ in terms of $u$ is symmetric and positive definite.
\end{lemma}

\begin{proof}
In case that $v_i,v_j$ is a pair of non-adjacent vertices, then
\[
\frac{\partial k_i}{\partial u_j}\,=\,\frac{\partial k_j}{\partial u_i}\,=\,0.
\]
In case that $[v_i,v_j]=e\in E$, there exist triangles $\triangle_1,\triangle_2$ adjacent to $e$. Lemma \ref{L-4-1} gives
\[
\frac{\partial k_i}{\partial u_j}\,=\,-\frac{\partial \vartheta_{i}^{\triangle_1}}{\partial u_j}-\frac{\partial \vartheta_{i}^{\triangle_2}}{\partial u_j}
\,=\,-\frac{\partial \vartheta_{j}^{\triangle_1}}{\partial u_i}-\frac{\partial \vartheta_{j}^{\triangle_2}}{\partial u_i}
\,=\,\frac{\partial k_j}{\partial u_i}\,\leq\,0.
\]
Hence the Jacobian matrix of $Th(\Theta,\cdot)$ in terms of $u$ is symmetric.

Meanwhile, a simple computation shows
\[
\frac{\partial k_i}{\partial u_i}\,=\,-\sum\nolimits_{\triangle\in F(\{v_i\})}\frac{\partial \vartheta_{i}^{\triangle}}{\partial u_i}\,>\,0.
\]
It follows that
\[
\bigg|\frac{\partial k_i}{\partial u_i}\bigg|-\sum\nolimits_{j\neq i}\bigg|\frac{\partial k_i}{\partial u_j}\bigg|\,=\,\sum\nolimits_{j=1}^{|V|}\frac{\partial k_i}{\partial u_j}\,=\,\sum\nolimits_{j=1}^{|V|}\frac{\partial k_j}{\partial u_i}\,=\,\sum\nolimits_{\triangle\in F(\{v_i\})}\frac{\partial}{\partial u_i}\mathrm{Area}(\triangle)\,>\,0.
\]
The Jacobian matrix is strictly diagonally dominant and hence is positive definite.
\end{proof}

Now it is ready to prove Theorem \ref{T-1-5}.

\begin{proof}[\textbf{Proof of Theorem \ref{T-1-5}}]
Consider the following $1$-form
\[
\omega\,=\,\sum\nolimits_{i=1}^{|V|}k_i \mathrm{d} u_i.
\]
Because $\partial k_i/\partial u_j=\partial k_j/\partial u_i$, it is easy to see  $\omega$ is closed. Thus the function
\[
\Phi(u)\,=\,\int_{u(0)}^u \omega
\]
is well-defined and does not depend on the particular choice of a
piecewise smooth arc in $\mathbb R^{|V|}_{-}$ from an initial point $u(0)$ to $u$. Note that
 the Hessian of $\Phi(\cdot)$ is equal to the Jacobian of $Th(\Theta,\cdot)$, which is positive definite by Lemma \ref{L-4-3}. That means $\Phi$ is a strictly convex function of $u$.

From Thurston's construction, the demanded  circle pattern is related to the critical point of $\Phi$. Because $\Phi$ is strictly convex, the critical point must be unique. We thus finish the proof.
\end{proof}

\begin{remark}\label{R-4-4}
Under the condition \textbf{(R1)}, the above analysis shows that the Jacobian matrix of  $Th(\Theta,\cdot)$ is non-singular. It follows from the Implicit Function Theorem that the demanded radius vector $r_\ast$  depends on $\Theta$ smoothly.
\end{remark}

\begin{remark}\label{R-4-5}
Based on observations in above parts, Xu \cite{Xu} recently studied rigidity of circle patterns with inversive distances. His work generalizes the results of Guo-Luo \cite{Guo-Luo} and Theorem \ref{T-1-5}.
\end{remark}

\begin{remark}\label{R-4-6}
Given an initial point, let it evolve with the negative gradient flow of $\Phi$. Following  Chow-Luo \cite{Chow-Luo}, one can show that the flow converges exponentially fast to the demanded circle pattern vector. See the recent work of Ge-Hua-Zhou \cite{Ge-Hua-Zhou} for details.
\end{remark}

\begin{proof}[\textbf{Proof of Theorem \ref{T-1-7}}]
The existence is an immediate consequence of Theorem \ref{T-1-2}.

For rigidity, suppose that $e_1,e_2,e_3$ form the boundary of a triangle of $\mathcal T$. According to the condition, we have
\[
\Theta(e_1)+\Theta(e_2)+\Theta(e_3)\ <\ \pi.
\]
Recall that $I(e_i)=\cos\Theta(e_i)$ for $i=1,2,3$. By Proposition \ref{P-2-7}, one obtains
\[
I(e_1)+I(e_2)I(e_3)\,\geq\, 0,\;\;
I(e_2)+I(e_3)I(e_1)\,\geq\, 0,\;\;
I(e_3)+I(e_1)I(e_2)\,\geq\, 0.
\]
It follows from Theorem \ref{T-1-5} that the rigidity holds.
\end{proof}

\section{Finiteness}\label{S-5}
To obtain the proof Theorem \ref{T-1-8}, some knowledge on Teichm\"{u}ller theory is needed. One refers to  \cite{Ahlfors,Imayoshi-Taniguchi} for basic background. See also \cite{Buser} from the geometric and topological viewpoint.

\subsection{Configuration spaces}
Recall that $S$ is an oriented closed surface of genus $g>1$. Denote by $\mathcal T(S)$ the Teichm\"{u}ller space of $S$, which parameterizes the equivalence classes of marked hyperbolic metrics on $S$. Here a marking is an isotopy class of orientation-preserving homeomorphism from $S$ to itself. And two marked hyperbolic metrics $\mu,\mu'\in \mathcal T(S)$ are equivalent if there exists an isometry $\phi: (S,\mu)\to(S,\mu')$ isotopic to the identity map. $\mathcal T(S)$ admits the structure of a smooth manifold of dimension $6g-6$.

Let $p_i,p_j\in S$ be two points and let $\gamma$ be a simple curve with endpoints $p_i,p_j$. For any $\mu\in \mathcal T(S)$, on $(S,\mu)$
there exists a unique geodesic in the isotopy class $[\gamma]$ with the endpoints $p_i,p_j$. Let $ d_{[\gamma]}(\mu,p_i,p_j)$ denote the length of this geodesic. Under the above manifold structure, $d_{[\gamma]}(\mu, p_i,p_j)$ depends on $\mu$ smoothly. See Buser's monograph \cite[Chap. 6]{Buser} for details.

We endow $S$ with a smooth structure and define a space $Z=\mathcal T(S)\times(S\times \mathbb R_{+})^{|V|}$.
Then $Z$ is a smooth manifold of dimension
\[
\dim\ (Z)\ =\ 6g-6+3|V|.
\]
Note that
\[
 2|E|\ =\ 3|F|.
\]
Therefore,
\[
\begin{aligned}
\dim\ (Z) \ =\ &6g-6+3|V|-\big(2|E|-3|F|\big)\\
 =\ &6g-6+3\big(|V|-|E|+|F|\big)+|E|\\
=\ &|E|.
\end{aligned}
\]
A point $z=(\mu,p_1,r_1,\cdots,p_{|V|},r_{|V|})\in Z$
is called a configuration, since it gives a choice of a marked hyperbolic metric $\mu$ on $S$ and assigns each vertex $v_i$ an oriented circle $C_i$  on $(S,\mu)$, where $C_i$ denotes the circle centered at $p_i$ with radius $r_i$. To summarise, a configuration $z$ gives a marked hyperbolic metric $\mu$ on $S$ together with a circle pattern $\mathcal P$ on $(S,\mu)$. Briefly, we say $z$ gives a circle pattern pair $(\mu,\mathcal P)$.

Let $Z_\mathcal T\subset Z$ denote the subspace of configurations that give $\mathcal T$-type circle pattern pairs. More precisely, $z\in Z_{\mathcal T}$ if and only if there exists a geodesic triangulation $\mathcal T(z)$ of $(S,\mu)$ such that $\mathcal T(z)$ is isotopic to $\mathcal T$ and the vertices of $\mathcal T(z)$ are  $p_1,p_2,\cdots,p_{|V|}$. Clearly, $Z_\mathcal T$ is open in $Z$ and thus is a smooth manifolds of dimension $|E|$.

For each $e=[v_i,v_j]\in E$, the inversive distance $I(e,z)$ is defined by
\[
I(e,z)\ =\ \frac{\cosh d_{[\gamma_e]}(\mu,p_i,p_j)-\cosh r_i\cosh r_j}{\sinh r_i\sinh r_j},
\]
where $[\gamma_e]$ represents the isotopy class of the edge $e$. Let $Z_{\mathcal {P}}\subset Z_\mathcal T$ be the subspace of configurations $z$ for which
\[
-1\,<\,I(e,z)\,<\,1,\;\forall\, e\in E.
\]
Because $I(e,z)$ is continuous, $Z_{\mathcal P}$ is open in $Z_{\mathcal T}$, which implies $Z_{\mathcal P}$ is a smooth manifold of dimension $|E|$. Obviously, a configuration $z\in Z_{\mathcal P}$ gives a $\mathcal T$-type circle pattern with the exterior intersection angle $\Theta(e,z)\in (0,\pi)$ for each $e\in E$, where
\[
\Theta(e,z)\ =\ \arccos I(e,z).
\]
One establishes the following smooth map
\[
\begin{aligned}
\mathcal{E}v:\quad &Z_{\mathcal P} &\longrightarrow \quad &\quad\;  Y:=(0,\pi)^{|E|}\\
&\;z\;  &\longmapsto \quad  & \big(\Theta(e_1,z),\Theta(e_2,z),\cdots\big).
\end{aligned}
\]

\subsection{Rigidity from the viewpoint of Sard's Theorem}
Note that a regular point of the above map gives a circle pattern pair which is locally determined up to isometry by its exterior intersection angles. Sard's Theorem indicates that local rigidity is a generic property.

\begin{proof}[\textbf{Proof of Theorem \ref{T-1-8}}]
Set $W_0=\mathcal Ev(Z_{0})$, where $Z_{0}\subset Z_{\mathcal P}$ denotes the set of critical points of $\mathcal Ev$. Owing to Sard's Theorem (Theorem \ref{T-6-2}), $W_0$ has zero measure. Meanwhile, the boundary set $\partial W$ of $W$ also has zero measure. Thus $W\setminus (W_0\cup\partial W)$ is a subset of $W$ with full measure. Note that
\[
\dim\, (Z_{\mathcal P})\,=\,\dim\, (Y)\,=\,|E|.
\]
For each $\Theta\in W\setminus (W_0\cup\partial W)\subset Y$, the Regular Value Theorem (Theorem \ref{T-6-1}) implies that $\mathcal Ev^{-1}(\Theta)$ is a discrete set. Combining with the following Proposition \ref{P-5-1}, this set must be finite. Thus the theorem is proved.
 \end{proof}

\begin{proposition}\label{P-5-1}
Let $\Theta$ be as above. Then $\mathcal Ev^{-1}(\Theta)$ is compact.
\end{proposition}

\begin{proof}
It suffices to show any sequence $\{z_n\}\subset \mathcal Ev^{-1}(\Theta)$ contains a convergent subsequence in $Z_{\mathcal P}$. Let $(\mu_n, \mathcal P_n)$ be the circle pattern pair give by $z_n$.

First we claim that there exists a subsequence $\{z_{n_k}\}$  converging to a point $z_\ast\in Z$. In view of the topology of $Z=\mathcal T(S)\times(S\times \mathbb R_{+})^{|V|}$, one needs to check the following two properties:
\begin{itemize}
\item[$1.$] No circle in $\mathcal P_n$ degenerates to a point or becomes infinitely large.
\item[$2.$] The sequence $\{\mu_n\}$ is included in a compact subset of $\mathcal T(S)$.
\end{itemize}

Because $\Theta$ satisfies \textbf{(C1)}, \textbf{(C2)}, similar arguments to the proof of Lemma \ref{L-3-1} show that the first property holds.

For the second property, suppose it is not true. By Fenchel-Nielsen's coordinates \cite{Imayoshi-Taniguchi,Buser}, that means there exists at least one simple closed geodesic $\gamma_n$ in $(S,\mu_n)$ whose length $\ell({\gamma_n})$ tends to zero or infinity. One claims that the diameter of $(S,\mu_n)$ tends to infinity. Indeed, if $\ell({\gamma_n})\to +\infty$, the statement trivially holds. If $\ell({\gamma_n})\to 0$, it follows from the Collar Theorem \cite[Chap. 4]{Buser} that there is an embedding cylinder domain $Cy(\gamma_n)$ in $(S,\mu_n)$ such that
\[
Cy(\gamma_n)\ =\ \big\{p\in(S,\mu_n)\,|\,\dist(p,\gamma_n)\leq d_n\big\},
\]
where $d_n>0$ satisfies
\[
\sinh\big(\ell(\gamma_n)/2\big)\sinh d_n\,=\,1.
\]
Note that $d_n\to\infty$ as $\ell(\gamma_n)\to 0$, which also implies the diameter of $(S,\mu_n)$ tends to infinity. On the other hand, because the radius of every circle in $\mathcal P_n$ is bounded from above, the diameter of $(S,\mu_n)$ must be bounded from above. This leads to a contradiction.

Next we show $z_\ast\in Z_{\mathcal T}$. Denote by $\mathcal T_{n_k}$ the geodesic triangulation for $\mathcal P_{n_k}$.
One needs to check that no triangle of $\mathcal T_{n_k}$ becomes infinitely large or degenerates to a point or a line segment. Using Lemma \ref{L-2-4}, this follows from the first property.

To finish the proof, it remains to prove $z_\ast\in Z_{\mathcal P}$. Because $\{z_{n_k}\}\subset \mathcal Ev^{-1}(\Theta)\subset Z_\mathcal P$, for any $e\in E$, we have $\Theta(e,z_{n_k})=\Theta(e)$. As $n_k\to\infty$, one obtains $\Theta(e,z_\ast)=\Theta(e)\in(0,\pi)$, which implies $z_\ast\in Z_{\mathcal P}$.
\end{proof}

\section{Appendix}\label{S-7}
 In this section we give a simple introduction to some results on manifolds, especially the topological degree theory. One refers to \cite{Guillemin, Hirsch,Milnor} for more background.

 Let $M, N$ be two smooth manifolds. A point $x\in M$ is called a critical point for a $C^1$ map $f: M\to N$ if the tangent map  $df_x: T_x M\to T_{f(x)}N$ is not surjective. Let $C_f$  denote the set of critical points of $f$. And $N\setminus f(C_f)$ is defined to be the set of regular values of $f$.

\begin{theorem}[Regular Value Theorem]\label{T-6-1}
Let $f: M\to N$ be a $C^r$ ($r\geq 1$) map, and let $y\in N$ be a regular value of $f$. Then $f^{-1}(y)$ is a closed $C^r$ submanifold of $M$.
If $y\in im(f)$, then the  codimension of $f^{-1}(y)$ is  equal to the dimension of $N$.
 \end{theorem}

\begin{theorem}[Sard's Theorem]\label{T-6-2}
Let $M, N$ be manifolds of dimensions $m,n$ and $f:M\to N$ be a $C^r$ map. If $r\geq\max\{1,m-n+1\}$, then $f(C_f)$ has zero measure in $N$.
\end{theorem}

Now assume that $M,N$ are oriented  manifolds of the same dimensions. Let $\Lambda \subset M$ be a relatively compact open subset. Namely, $\Lambda$ has compact closure in $M$. For a point $y\in N$, we use $f\mbox{\large $\pitchfork$}_{\Lambda}y$ to denote that $y$ is a regular value of the restriction map $f: \Lambda \to N$.

Suppose $f\in C^{0}(\bar{\Lambda},N)\cap C^{\infty}(\Lambda,N)$, $f\mbox{\large $\pitchfork$}_{\Lambda}y$, and $f(\partial \Lambda)\in N\setminus\{y\}$. If $\Lambda \cap f^{-1}(y)$ is empty, the topological degree  $\deg(f,\Lambda,y)$ is defined to be zero. If $\Lambda \cap f^{-1}(y)$ is non-empty, the Regular Value Theorem implies that it consists of finite points.

\begin{definition}
Suppose $\Lambda \cap f^{-1}(y)=\{x_1,\cdots,x_s\}$. Set
\[
\deg(f,\Lambda,y)\,=\,\sum\nolimits_{i=1}^s \mathrm{sgn}(f,x_i).
\]
Here $\mathrm{sgn}(f,x_i)=+1$, if
the tangent map $df_{x_i}:T_{x_i} M \to T_{y}N$  preserves the orientation;
Otherwise, $\mathrm{sgn}(f,x_i)=-1$.
\end{definition}

In what follows, $I$ represents the interval $[0,1]$.
\begin{proposition}\label{P-6-4}
Suppose  $f_{i}\in C^{0}(\bar{\Lambda},N)\cap
C^{\infty}(\Lambda,N)$, $f_{i}\mbox{\large $\pitchfork$}_{\Lambda}y
$ and $f_{i}(\partial \Lambda)\subset N\setminus\{y\}$, $i=0,1$. If there
exists a homotopy $
H \in C^{0}({I\times \bar\Lambda},N)$
such that
\begin{itemize}
\item[$(i)$]$H(0,\cdot)=f_0(\cdot)$,  $H(1,\cdot)=f_1(\cdot)$,
\item[$(ii)$] $H(I\times\partial \Lambda)\subset N\setminus\{y\}$,
 \end{itemize}
 then
\[
\deg(f_{0},\Lambda,y)\ =\ \deg(f_{1},\Lambda,y).
\]
\end{proposition}

The following lemma is a consequence of Sard's Theorem.
\begin{lemma}\label{L-6-5}
Given $f\in C^{0}(\bar{\Lambda},N)$ and $y\in N$, if $f(\partial \Lambda)\subset N\setminus\{y\}$, then there exists $g\in C^{0}(\bar{\Lambda},N)\cap
C^{\infty}(\Lambda,N)$ and $H\in C^{0}(I\times\bar{\Lambda},N)$ such
that
\begin{itemize}
\item[$(i)$] $g\mbox{\large $\pitchfork$}_{\Lambda} y$,
\item[$(ii)$] $H(0,\cdot)=f(\cdot)$, $H(1,\cdot)=g(\cdot)$,
\item[$(iii)$] $H(I\times\partial \Lambda)\subset N\setminus\{y\}$.
\end{itemize}
\end{lemma}

It is ready to define the topological degrees for general continuous maps.
\begin{definition}
For any $f\in C^{0}(\bar{\Lambda},N)$ and $y\in N$ that satisfy $f(\partial \Lambda)\subset N\setminus \{y\}$,  one defines
\[
\deg(f,\Lambda,y)\ =\ \deg(g,\Lambda,y),
\]
where $g$ is given by Lemma \ref{L-6-5}.
\end{definition}

By Proposition \ref{P-6-4}, $\deg(f,\Lambda,y)$ is well-defined and does not depend on the particular choice of $g$. Below are some properties of topological degrees.

\begin{theorem}\label{T-6-7}
Suppose $f_{i}\in C^{0}(\bar{\Lambda},N)$ satisfies $f_{i}(\partial \Lambda)\subset N\setminus \{y\}$, $i=0,1$. If there exists a homotopy $H\in C^{0}(I\times\bar{\Lambda},N)$ such that
\begin{itemize}
\item[$(i)$] $H(0,\cdot)=f_{0}(\cdot)$, $H(1,\cdot)=f_{1}(\cdot)$,
\item[$(ii)$] $H(I\times\partial \Lambda)\subset N\setminus \{y\}$,
\end{itemize}
then
\[
\deg(f_{0},\Lambda,y)\ =\ \deg(f_{1},\Lambda,y).
\]
\end{theorem}


\begin{theorem}\label{T-6-8}
If $\deg\big(f,\Lambda,y\big)\neq 0$, then  $\Lambda \cap
f^{-1}(y)\neq\emptyset$.
\end{theorem}

\section{Acknowledgement}
The first version of this paper considered Lemma \ref{L-2-4} under the condition that $\Theta_i+\Theta_j\leq \pi, \Theta_j+\Theta_k\leq \pi, \Theta_k+\Theta_i\leq \pi$. During a workshop on discrete and computational geometry at the Capital Normal University in Beijing,  Feng Luo, Xu Xu and Qianghua Luo suggested the current condition. The author is very grateful for their generous sharing. Part of this work was done when the author was visiting Rutgers University, he would like to thank for its hospitality. He also would like to thank  NSF of China
(No.11601141 and No.11631010) and China Scholarship Council (No. 201706135016) for financial support.

\noindent Ze Zhou, zhouze@hnu.edu.cn\\[2pt]
\emph{Institute of Mathematics, Hunan University, Changsha, 410082, P.R. China}
\end{document}